\documentclass[12pt,reqno]{amsart}
\usepackage{graphicx, amssymb, xcolor,pdfsync}
\usepackage[all,cmtip]{xy}
\numberwithin{equation}{section}
\usepackage{mathrsfs}
\usepackage{tikz}
\usepackage{tikz-cd}
\usepackage{faktor}

 \usepackage{mathpazo}
\usepackage{hyperref}
\usepackage{graphicx}
\usepackage{yhmath}
\usepackage{mathdots}
\usepackage{MnSymbol}

\hypersetup{
    colorlinks=true,       
    linkcolor=blue,          
    citecolor=blue,        
    filecolor=blue,      
    urlcolor=blue           
}
\usepackage{tikz}
\usetikzlibrary{matrix,arrows}

\usepackage[switch]{lineno}
\usepackage{yfonts}

\advance\hoffset by -0.9 truecm

\begin{document}
\newcommand{\s}{\vspace{0.2cm}}
\newcommand{\Sp}{\mbox{Sp}}
\newcommand{\sy}{\mbox{sym}}
\newcommand{\GP}{\mathscr{C}_p}
\newcommand{\FP}{\mathcal{F}_p}
\newcommand{\G}{\mathbb{Z}_p^2 \rtimes \mathbf{D}_3}
\newcommand{\Z}{\mathbb{Z}}
\newcommand{\D}{\mathbf{D}}
\newcommand{\HH}{\mathbb{H}}
\newcommand{\rr}{\mbox{rot}}

\newtheorem{theo}{Theorem}
\newtheorem{prop}{Proposition}
\newtheorem{coro}{Corollary}
\newtheorem{lemm}{Lemma}
\newtheorem{claim}{Claim}
\newtheorem{example}{Example}
\theoremstyle{remark}
\newtheorem*{rema}{\it Remark}
\newtheorem*{rema1}{\it Remark}
\newtheorem*{defi}{\it Definition}
\newtheorem*{theo*}{\bf Theorem}
\newtheorem*{coro*}{Corollary}
\newtheorem{hecho}[theo]{\bf Hecho}

\title[A generalisation of the pencil of  Kuribayashi-Komiya quartics]{A generalisation of the pencil of   Kuribayashi-Komiya quartics}
\date{}

\author{Valentina Moreno Vega and Sebasti\'an Reyes-Carocca}
\address{Departamento de Matem\'aticas, Facultad de Ciencias, Universidad de Chile, Las Palmeras 3425, Santiago, Chile}
\email{valentina.moreno.1@ug.uchile.cl, sebastianreyes.c@uchile.cl}

\thanks{Partially supported by ANID Fondecyt Grants 1220099 and 1230708, ANID Beca Mag\'ister 8964 and Facultad de Ciencias, Universidad de Chile}

\begin{abstract} 
The pencil of Kuribayashi-Komiya quartics $$
x^4 + y^4 + z^4 + t(x^2y^2 + x^2z^2 + y^2z^2)=0 \, \mbox{ where } t \in \bar{\mathbb{C}}
$$is a complex one-dimensional family of Riemann surfaces of genus three endowed with a group of automorphisms isomorphic to the symmetric group of order twenty-four. This pencil has been extensively studied from different points of view. 

This paper is aimed at studying, for each prime number  $p \geqslant 5$, the pencil of \textit{generalised Kuribayashi-Komiya curves} $\mathcal{F}_p$, given by the curves
$$x^{2p}+y^{2p}+z^{2p}+t(x^p y^p +x^p z^p +y^p z^p)=0\mbox{ where } t \in \bar{\mathbb{C}}.$$ 

We determine the full automorphism group $G$ of each smooth member $X \in \mathcal{F}_p$ and study the action of $G$ and of its subgroups on $X$. In particular, we show that no member of the pencil is hyperelliptic. As a by-product, we derive a classification of all those Riemann surfaces of genus $(p-1)(2p-1)$ that are endowed with a group of automorphisms isomorphic to the full automorphism group of the generic smooth member of $\mathcal{F}_p.$
\end{abstract}
\maketitle
\thispagestyle{empty}

\section{Introduction}

The problem of the classification of (smooth complex projective) algebraic curves or, equivalently, compact Riemann surfaces according to their symmetries is a classical subject in algebraic geometry.  One of the first questions asked about these objects was if they had any automorphisms, and if so how many of them. Whereas the algebraic curves of genus zero and one have infinitely many automorphisms, if $X$ is an algebraic curve of genus $g \geqslant 2$ then $X$ has finitely many automorphisms, and$$|\mbox{Aut}(X)|\leqslant 84(g-1).$$This is a well-known result due to Schwarz \cite{Sch1879} and Hurwitz \cite{Hu}. Ever since, a classical problem has been the determination of the full automorphism group of a given algebraic curve. Besides, a seminal result in the area claims that each abstract finite group can be realised as a group of automorphisms of some algebraic curve \cite{Gre}. 

\s

Articles aimed at studying group actions of special classes of groups can be found in the literature in plentiful supply. We refer to \cite{B22} and \cite{BCsurvey} for an up-to-date treatment of this topic. We refer to Kuribayashi-Komiya \cite{KK79} and Broughton \cite{B91} for the automorphism group of algebraic curves genus two and three, and Bartolini, Costa and Izquierdo \cite{BCIrac} for genus four and five. 

With the advent of significant advancements in computer algebra systems, lists of groups of automorphisms for small genera are available; see, for instance \cite{conder}, \cite{K} and \cite{jen}. 

\s

Kuribayashi and Komiya in \cite{KK77} introduced the pencil of quartics $\mathcal{F}_2$ given by the algebraic curves $$x^{4} + y^{4} + z^{4} + t(x^2 y^2 + x^2 z^2 + y^2 z^2) = 0, \mbox{ where } t \in \bar{\mathbb{C}}.$$

Up to four special curves that are singular, the members this pencil are non-hyperelliptic Riemann surfaces of genus three and have a group of automorphisms isomorphic to the symmetric group of order twenty-four, generated by the transformations $$[x:y:z]\mapsto [x:z:-y] \,\, \mbox{ and } \,\, [x:y:z]\mapsto [z:y:x].$$

This group agrees with the full automorphism group of each smooth member of $\mathcal{F}_2$ with only two celebrated exceptions: the  Fermat quartic and  the Klein  quartic $$x^4 + y^4 + z^4 =0 \,\, \mbox{ and } \,\,x^3y+y^3z+z^3x=0,$$which are obtained by taking $t=0$ and $t=\tfrac{3}{2}(-1\pm 7\sqrt{-1})$ respectively.  

\s

This pencil (which is also known in the literature as the KFT family) and distinguished members of it have beautiful properties and have been extensively studied from different points of view (such as its Fuchsian uniformisation, the behaviour of its limits points, the Weierstrass points, the field of moduli, the moduli of the intermediate quotients, the corresponding Jacobian varieties, maximal tangency lines, among others aspects). See, for instance, \cite{CGA17}, \cite{dolga}, \cite{GT02}, \cite{rubenkft}, \cite{ik}, \cite{KS79}, \cite{MZ},  \cite{RCS}, \cite{KFT} and \cite{W1895}.

\s

Let $p \geqslant 3$ be a prime number. This article is aimed at  studying a natural generalisation this pencil. Concretely, we study the {\it{generalised Kuribayashi-Komiya pencil of degree $2p$}} defined as the algebraic curves in $\mathbb{P}^2$ given by
$$x^{2p} + y^{2p} + z^{2p} + t(x^p y^p + x^p z^p + y^p z^p) = 0 \mbox{ where }t \in \bar{\mathbb{C}}.$$ We denote this pencil by $\mathcal{F}_p$ and the subset formed by its smooth members by $\mathcal{F}_p^*$.

\s

The results of this paper, which will be stated in \S\ref{state}, can be succinctly summarised as follows. We study the full automorphism group $G$ of the Riemann surfaces in $\FP^*$ and prove that, generically, this group is generated by the transformations
$$\left(\begin{smallmatrix}
\zeta_p&0&0\\
0&1&0\\
0&0&1\\
\end{smallmatrix}\right), \,\, \left(\begin{smallmatrix}
1&0&0\\
0&\zeta_p&0\\
0&0&1\\
\end{smallmatrix}\right), \,\, \left(\begin{smallmatrix}
0&1&0\\
0&0&1\\
1&0&0\\
\end{smallmatrix}\right), \,\, \left(\begin{smallmatrix}
0&0&1\\
0&1&0\\
1&0&0\\
\end{smallmatrix}\right) \,\, \mbox{ where }\zeta_p=\mbox{exp}(2 \pi i / p),$$and is isomorphic to the semidirect product \begin{equation}\label{pg8}\G=\langle a, b, s, r :  a^{p},b^{p},r^3,s^2,(sr)^2,[a,b], rar^{-1}ab, rbr^{-1}a^{-1}, sasab, [s,b]\rangle.\end{equation}
We then introduce a {\it family of Riemann surfaces} (see \S\ref{prelis} for a precise definition of {\it family}) in the singular locus of the moduli space that contains $\FP^*$,  and employ the combinatorial theory of Fuchsian groups to derive further properties of the members of $\FP^*,$ such as the description of the action of $G$ and of its subgroups. In particular, we deduce that $\mathcal{F}_p^*$ does not intersect the hyperelliptic locus. After that, we study the {\it exceptional members} of $\FP^*,$ namely, those members of the pencil that admit more automorphisms than the generic member. We prove that there is only one such an exceptional member: the Fermat curve  $$x^{2p}+y^{2p}+z^{2p}=0.$$ We also observe that no member of $\FP^*$ generalises the role played by the Klein quartic in $\mathcal{F}_2.$ Finally, as an application of our results, we provide a classification of all compact Riemann surfaces --or, algebraic curves-- of genus $(p-1)(2p-1)$ that are endowed with a group of automorphisms isomorphic to the semidirect product $\G$ with presentation \eqref{pg8}. This last result contributes to the broader program of classifying group actions on Riemann surfaces.
\s

It is worth mentioning that different generalisations of $\mathcal{F}_2$ can be found in the literature. See, for instance,  \cite{SE14} and \cite{H00}.

\s

This article is organised as follows. In \S\ref{prelis} we briefly review the basic background and introduce some notation. The statement of the results is given in \S\ref{state}, and the proofs of them are given in  \S\ref{dems}. We end the article by pointing out two remarks in \S\ref{rems}.

\section{Preliminaries and Notations}\label{prelis}

Let $X$ be a Riemann surface of genus $g \geqslant 2$ endowed with an action $\psi: G \to \mbox{Aut}(X)$ of a finite group $G.$ Let $\pi_G : X \to X/G$ be the regular covering map given by the action of $G \cong \psi(G)$ on $X$, and assume that $\pi_G$ ramifies over $r$ values $y_1, \ldots, y_r$. The {\it signature} of the action of $G$ is  $$\sigma=(h; m_1, \ldots, m_r) \in \mathbb{Z}^+_0 \times (\mathbb{Z}_{\geqslant 2})^r$$where $h$ is the genus of $X/G$ and  $m_i$ is the order of the $G$-stabiliser of $x_i \in X$ where $\pi_G(x_i)=y_i$.

Let $\Gamma$ be a Fuchsian group such that $X \cong \mathbb{H}/\Gamma.$ Riemann existence theorem ensures that $G$ acts on $X$ if and only if there is a Fuchsian group $\Delta$ and there is an epimorphism  \begin{equation}\label{enc}\theta : \Delta \to G \mbox{ such that } \mbox{ker}(\theta)=\Gamma.\end{equation} The genus $g$ of $X$ is related to $\sigma$ by the {\it Riemann-Hurwitz formula} $$2(g-1)=|G|(2(h-1)+\Sigma_{j=1}^r(1-1/m_j)).$$

The epimorphism $\theta$ represents the action and is called a {\it surface-kernel epimorphism}. It is said that the signature of $\Delta$ is $\sigma$. It is well known that $\Delta$ has a canonical presentation   $$\langle \alpha_1, \beta_1, \ldots, \alpha_h, \beta_h, x_1, \ldots, x_r : \Pi_{i=1}^h\alpha_i\beta_i\alpha_i^{-1}\beta_i^{-1}\Pi_{j=1}^rx_j=x_1^{m_1}=\cdots =x_r^{m_r}=1\rangle.$$The elements of finite order of $\Delta$ are those that are conjugate to powers of $x_i.$ For simplicity, when the Fuchsian group is clear enough, we identify $\theta$ as in \eqref{enc} with the {\it generating vector} $$\theta=(\theta(\alpha_1), \ldots, \theta(\alpha_h), \theta(\beta_1), \ldots, \theta(\beta_h), \theta(x_1), \ldots, \theta(x_r)) \in G^{2h+l}.$$

Let $G'$ be a supergroup of $G.$ The action of $G$ on $X$ is said to {\it extend} to an action of $G'$ if there is a Fuchsian group $\Delta'$ with $\Delta \leqslant \Delta'$, both having the same Teichm\"{u}ller dimension, and there is a surface-kernel  epimorphism $$\Theta: \Delta' \to G' \, \mbox{ with }\,   \Theta|_{\Delta}=\theta.$$

An action is called {\it maximal} if it cannot be extended in the previous sense. The list of pairs of signatures that may allow having an extension as before was found by Singerman in  \cite{Sing72}. In addition, for each such a pair, Bujalance, Cirre and Conder in  \cite{BCC} provided necessary and sufficient conditions under which such an extension is realised; this condition is given in terms of the surface-kernel epimorphism representing the action. See also \cite{Ries93}.  

\s

Let $Y$ be a  Riemann surface endowed with an action $\psi':G \to \mbox{Aut}(Y)$. If there exists an orientation-preserving homeomorphism  \begin{equation} \label{lapiz22}\phi: X \to Y \mbox{ such that } \phi \psi(G) \phi^{-1}=\psi'(G)\end{equation}then it is said that the actions $\psi$ and $\psi'$ are {\it topologically equivalent}. This relation can be read off from surface-kernel epimorphisms as follows. If we write $X/{G} \cong \mathbb{H}/{\Delta}$ and $Y/G \cong \mathbb{H}/{\Delta}'$ then each homeomorphism $\phi$ as in \eqref{lapiz22}  induces  a group isomorphism $\phi^*: \Delta \to \Delta'$. Hence, we may assume $\Delta=\Delta'$. We denote the subgroup of $\mbox{Aut}(\Delta)$ consisting of the automorphisms $\phi^*$ by $\mathscr{B}$. Following \cite{B90}, the surface-kernel epimorphisms $\theta : \Delta \to G$ and $ \theta': \Delta \to G$ representing two actions of $G$  are topologically equivalent (denoted by $\theta \sim \theta'$) if and only if $$\mbox{there exist } a \in \mbox{Aut}(G) \mbox{ and } \phi^* \in \mathscr{B} \mbox{ such that } \theta' = a\circ\theta \circ \phi^*.$$

A {\it family of Riemann surfaces} of genus $g$ is the locus $$\mathcal{F}=\mathcal{F}(G, \sigma) \subset \mbox{Sing}(\mathscr{M}_g)$$ formed by the points representing isomorphism classes of Riemann surfaces endowed with an action of $G$ with signature $\sigma$. According to \cite{B90}, the interior of the family $\mathcal{F}$, if non-empty, consists of those Riemann surfaces whose (full) automorphism group is isomorphic to $G$ and is formed by finitely many components or {\it equisymmetric strata}, each one of them being a closed irreducible (non necessarily smooth) subvariety of  $\mathscr{M}_g$ (see also \cite{GG92}).

\section{Statement of the results}\label{state}

Let $p \geqslant 2$ be a prime number. We denote by $X_{p,t}$ the member of the pencil $\FP$ given by 
$$x^{2p}+y^{2p}+z^{2p}+t(x^p y^p +x^p z^p +y^p z^p)=0 \mbox{ where } t \in \bar{\mathbb{C}}.$$

A routine computation shows that $X_{p,t}$ is smooth if and only if $t \in \bar{\mathbb{C}}-\{\infty, -1, \pm 2\},$ and therefore in such cases the curve $X_{p,t}$ has a structure of compact   Riemann surface of genus $(p-1)(2p-1).$ The four singular members of $\FP$ will be briefly discussed in \S\ref{singular}.

\s

We recall that $\FP^*$ stands for the subset of $\FP$ formed by its smooth members.

\begin{prop} \label{noprimo1}
Let $p \geqslant 2$ be a prime number and let $X_{p,t} \in \FP^*.$ The following statements hold. 

\begin{enumerate}
\item If $\zeta_p=\mbox{exp}(2 \pi i/p)$ then the automorphisms of $\mathbb{P}^2$ given by $$\bf{a}=\left(\begin{smallmatrix}
\zeta_p&0&0\\
0&1&0\\
0&0&1\\
\end{smallmatrix}\right), \,\, \bf{b}=\left(\begin{smallmatrix}
1&0&0\\
0&\zeta_p&0\\
0&0&1\\
\end{smallmatrix}\right), \,\, \bf{r}=\left(\begin{smallmatrix}
0&1&0\\
0&0&1\\
1&0&0\\
\end{smallmatrix}\right), \,\, \bf{s}=\left(\begin{smallmatrix}
0&0&1\\
0&1&0\\
1&0&0\\
\end{smallmatrix}\right)$$restrict to $X_{p,t}$ and generate a group of automorphisms  $G_p$ of it isomorphic to a semidirect product $\G$ with presentation
\begin{equation}\label{pg}\langle a, b, s, r :  a^{p},b^{p},r^3,s^2,(sr)^2,[a,b], rar^{-1}ab, rbr^{-1}a^{-1}, sasab, [s,b]\rangle.\end{equation}
\item The signature of the action of $N_p = \langle \mathbf{a,b}\rangle \cong \mathbb{Z}_p^2$ and of $G_p$ on $X_{p,t}$ is $(0; p, \stackrel{6}{\ldots}, p)$ and $(0;2,2,3,p)$.

\s

\item If $p \neq 3$ then $\mbox{Aut}(X_{p,t})=G_p$, up to finitely many exceptional values of $t$.
\end{enumerate}
\end{prop}

Along the paper, we shall write $G_p=\langle\mathbf{r,s,a,b} \rangle \leqslant \mbox{Aut}(X_{p,t}),$ and denote by $\G$ the semidirect product with presentation \eqref{pg}.

\s

Let $\GP$ be the family formed by the compact Riemann surfaces endowed with a group of automorphisms isomorphic to $G=\G$ acting with signature $(0; 2,2,3,p).$ Note that $$\GP \subset \mbox{Sing}(\mathscr{M}_{(p-1)(2p-1)})$$has complex dimension one. Clearly, by definition, we have that $\FP^* \subset \GP$. Let $X \in \GP$ and let $$\Gamma=\langle x_1, x_2, x_3, x_4 : x_1^2=x_2^2=x_3^3 = x_4^p= x_1x_2x_3x_4 =1\rangle$$be a Fuchsian group of signature $(0;2,2,3,p)$. As mentioned in the introduction, Riemann's existence theorem coupled with Proposition \ref{noprimo1} imply that there exists a group epimorphism 
 \begin{equation}\label{lapiz}\Theta: \Gamma \to G=\G\end{equation} such that $\mathbb{H}/\mbox{ker}(\Theta) \cong X$  and $\mathbb{H}/\Gamma \cong X/G$.
 
 \s
 
 It is well-known that the family $\mathscr{C}_2$ consists of only one equisymmetric stratum (see, for instance, \cite[Table 5]{B91}), whereas the family $\mathscr{C}_3$ consists precisely of two equisymmetric strata (as can be obtained by using the routines implemented in  \cite{BRR13}). As the automorphism groups of compact Riemann surfaces of genus ten are known, from now on we assume $p \geqslant 5.$

\begin{prop}\label{thetak}
Let $p \geqslant 5$ be a prime number. The action of $G=\G$ (with presentation \eqref{pg}) on $X \in \GP$  is represented by the surface-kernel epimorphism $$\theta_k : \Gamma \to G \, \mbox{ given by } \, \theta_k= (s, sr, a^k b^{k-1}r^2, ab^k)$$for some $k \in \{1, \ldots, p-2\}.$ Furthermore:
\begin{enumerate} 
\item if $p\equiv 1\mbox{ mod }3$ then $\mathscr{C}_p$ consists of at most $\frac{p+1}{2}$  equisymmetric strata, and
\item if $p\equiv 2\mbox{ mod }3$  then $\mathscr{C}_p$ consists of at most $\frac{p-1}{2}$ equisymmetric strata.
\end{enumerate} 
\end{prop}

We denote by $\GP^k$ the equisymmetric stratum of $\GP$ determined by $\theta_k$.

\begin{prop} \label{fanta}
Let $p\geqslant 5$ be a prime number and let $X \in \GP^k$ where $k \in \{1, \ldots, p-2\}.$ For each subgroup $H$ of $G$, the genus $g_H$ of $X/H$ and the ramification data of the regular covering map $X \to X/H$ is summarised in the following table.

\s
{\tiny
\begin{center}
\begin{tabular}{|c | c |c |c|} 
 \hline
 $H$ & $\mbox{ order of }H$ & $g_H$ & $\mbox{ramification data}$ \\ 
 \hline\hline
   $\langle ab^l \rangle$ & $p$ & $2(p-1)$ & $(-)$ if $k\neq l$ \\
  &  & $p-1$ & $(p, \stackrel{2p}{\ldots}, p)$ if $k=l$ \\  [0.3ex]
 \hline
  $\langle s \rangle$ & $2$ & $(p-1)^2$ & $(2,\stackrel{2p}{\ldots},2)$ \\  [0.3ex]
 \hline
  $\langle b, s \rangle$ & $2p$ & $p-1$ & $(2, 2)$ if $k \neq 1$ \\ 
 &  & $\tfrac{p-1}{2}$ & $(2,2, p, \stackrel{p}{\ldots}, p)$ if $k = 1$ \\ [0.3ex]
 \hline
  $\langle a, b \rangle$ & $p^2$ & $0$ & $(p, \stackrel{6}{\ldots}, p)$ \\ [0.3ex]
 \hline
  $\langle r \rangle$ & $3$ & $\frac{(2p-1)(p-1)}{3}$ & $(3,3)$ \\ [0.3ex]
 \hline
  $\langle s, r \rangle$ & $6$ & $\frac{(p-1)(p-2)}{3}$ & $(2, \stackrel{2p}{\ldots},2, 3)$ \\ [0.3ex]
 \hline
  $\langle as \rangle$ & $2p$ & $p-1$ & $(2,2)$ if $k \neq 1$ \\ 

&  & $\frac{p-1}{2}$ & $(2,2,p, \stackrel{p}{\ldots}, p)$ if $k = 1$ \\ [0.3ex]
 \hline
 $\langle a,b \rangle \rtimes \langle r\rangle$ & $3p^2$ & $0$ & $(3,3,p,p)$ \\ [0.3ex]
 \hline
  $\langle a,b \rangle \rtimes \langle s\rangle$ & $2p^2$ & $0$ & $(2,2,p,p,p)$ \\ 
 \hline
\end{tabular}
\end{center}}
\end{prop}

\s

As a consequence of the proposition above, we obtain the following corollary.

\begin{coro} \label{pepsi} Let $p \geqslant 5$ be a prime number. If the automorphism group of $X \in \mathscr{C}_p$ is isomorphic to $\G$, then $X$ is not cyclic $n$-gonal for each $n \geqslant 2$. In particular, the generic member of $\mathscr{C}_p$ is non-hyperelliptic.
\end{coro}

\begin{rema}
As $\FP^* \subset \GP$, Proposition \ref{fanta} and Corollary \ref{pepsi} also hold for the  members of $\mathcal{F}_p^*.$
\end{rema}

We recall that  $\mathcal{F}_2$ has exactly two {\it exceptional members}, in the sense that they have more automorphisms than the general member: the Fermat quartic and the Klein quartic. It is clear that the Fermat curve $\mathbf{F}_{2p}$ 
 belongs to $\FP^*$ for each $p \geqslant 2$, and hence belongs to the family $\GP.$ 
 
 The following result locates $\mathcal{F}_p^*$ inside $\GP$ with the stratum $\GP^1.$

\begin{theo}\label{cig}
Let $p \geqslant 5$ be a prime number. Then $\mathbf{F}_{2p} \in \GP^1$ and $\mathcal{F}_p^* \subset \GP^1$.
\end{theo}

Following \cite[Theorem 1]{P14}, if $X$ is a complex smooth plane projective algebraic curve of degree $d \geqslant 8$, then $$|\mbox{Aut}(X)| \leqslant 6d^2,$$and the equality holds if and only if $X \cong \mathbf{F}_{d}$. It follows that for $X \in \FP^* $ we have that \begin{equation}\label{cot}\mbox{Aut}(X) \neq G_p \mbox{ then } |\mbox{Aut}(X)|=6p^2\lambda  \mbox{ for some } \lambda=2,3,4\end{equation}The case $\lambda=4$ corresponds to the Fermat curve $\mathbf{F}_{2p}$.

\begin{rema}
Observe that \eqref{cot} implies that two natural generalisations of the Klein quartic do not belong to $\FP^*$.
\s

{\bf 1.} No member of $\FP^*$ is isomorphic to the algebraic curve
$xy^{2p-1} + yz^{2p-1} + zx^{2p-1}=0,$ due to the fact that its automorphism group is isomorphic to a semidirect product of the form $\Z_{4p^2-6p+3} \rtimes \Z_3.$ See \cite[Theorem 5.3]{ABS21}.

\s

{\bf 2.} No member of $\FP^*$ is isomorphic to a Hurwitz curve. More generally, this fact also holds for $\mathscr{C}_p$. Indeed, if $p > 23$ then, following \cite{ruben2}, the subgroup $N=\langle {a,b}\rangle$ is the unique Sylow $p$-subgroup of $G$ and therefore there exists an exact sequence  $$1 \to N \to \mbox{Aut}(X) \to A = \mbox{Aut}(X)/N \to 1,$$where $A$ is a spherical group (see Proposotion \ref{fanta}).  Thus,  $|\mbox{Aut}(X)| =|A|p^2 \leq 60p^2 < 84(g-1).$
\end{rema}

The following result shows that the case $\lambda=3$ is not realised.
 
\begin{prop} \label{con3}
    Let $p \geqslant 5$ be a prime number. No member of $\mathscr{C}_p$ (and hence of $\FP^*$) has automorphism group of order $18p^2$. 
\end{prop}

We now deal with the case $\lambda=2$. Consider the following groups of order $12p^2.$
\begin{multline*}  H_1 = \langle  A, B, S, R, C : A^p,B^p,S^2,R^3,C^2,(SR)^2,  [A,B],[B,S], \\SASAB, RAR^{-1}AB,  RBR^{-1}A^{-1},  
    [C,A],[C,B],[S,C],[R,C]  \rangle \cong (\Z_p^2 \rtimes \D_3) \times \Z_2, \end{multline*}  \begin{multline*}  H_2 = \langle  A, B, S, R, C : A^p,B^p,S^2,R^3,C^2,(SR)^2,  [A,B],[B,S],\\SASAB, RAR^{-1}AB,  RBR^{-1}A^{-1},  
    (CA)^2,  (CB)^2, [S,C],[R,C]  \rangle \cong (\Z_p^2 \rtimes \D_3) \rtimes \Z_2. 
 \end{multline*}

\begin{prop}\label{prol2} Let $p \geqslant 5$  be a prime number different from 11. The family $\GP$ has exactly two members $Z_1$ and $Z_2$ with automorphism group of order $12p^2.$ These Riemann surfaces satisfy $$Z_i \in \mathscr{C}_p^i \mbox{ and } \mbox{Aut}(Z_i) \cong H_i \mbox{ for } i=1,2.$$  
Furthermore, $Z_1$ and $Z_2$ do not belong to $\mathcal{F}_p^*.$
\end{prop}

As a consequence of the propositions above, we obtain the following theorem.

\begin{theo}\label{fullg}Let $p \geqslant 5$  be a prime number different from 11. The automorphism group of the members of $\FP^*$ is given by \begin{displaymath}
\mbox{Aut}(X_{p,t}) \cong \left\{ \begin{array}{cc}
\mathbb{Z}_{2p}^2 \rtimes \mathbf{D}_3 & \textrm{if $X_{p,t} \cong \mathbf{F}_{2p}$}\\
\mathbb{Z}_{p}^2 \rtimes \mathbf{D}_3 & \textrm{otherwise }
\end{array} \right.
\end{displaymath}
\end{theo}

We end this section by providing a complete classification of all those compact Riemann surfaces of genus $(p-1)(2p-1)$ endowed with a group of automorphisms isomorphic to $\G.$ To state this classification, we need to introduce some notation. Let $\Delta$ be a Fuchsian group of signature 
$(0; 3, 2p, 2p)$ canonically presented as $$\langle y_1,y_2, y_3 : y_1^3=y_2^{2p}= y_3^{2p}=y_1y_2 y_3=1 \rangle,$$and consider the surface-kernel epimorphisms $\varphi_n: \Delta \to \G$ given by $$\varphi_n(y_1, y_2, y_3)=(a^{1-n}br^2,bsr, b^{n}sr^2) \mbox{ where } n \in \{1, \ldots, p-1\}.$$We denote by $Y_n$ the compact Riemann surface $\mathbb{H}/\mbox{ker}(\varphi_n)$ for each $n \in \{1, \ldots, p-1\}.$

\begin{theo}\label{clasi}
Let $p\geqslant 5$ be a prime number different from 11. If $W$ is a compact Riemann surface of genus $(p-1)(2p-1)$ with a group of automorphisms isomorphic to $\G$ then
\begin{enumerate}
\item $W \in \mathscr{C}_p$ or
\item $W$ is isomorphic to $Y_j$ for some $j \in \{ 2, \ldots, p-2\}$.
\end{enumerate} Furthermore, up to isomorphism, there are exactly $\tfrac{p-3}{2}$ Riemann surfaces in the latter case, they do not belong to $\GP$ and the automorphism group of each of them is isomorphic to $\G.$
\end{theo}

We anticipate that there is an intimate relation between the Riemann surfaces appearing in the theorem above and the so-called generalised Fermat curves; we shall discuss that in \S\ref{CFG}.

\section{Proofs}\label{dems}

\subsection{Proof of Proposition \ref{noprimo1}}
 Let $p \geqslant 2$ be a prime number and let $X_{p,t} \in \FP^*.$ The first statement is obtained after a routine computation. We denote by $$\pi : X_{p,t} \to Y_{p,t}:=X_{p,t}/N_p$$ the regular covering map given by the action of $N_p $ on $X_{p,t}$. The branch points of  $\pi$ agree with the points of $X_{p,t}$ that are fixed by the nontrivial elements of $N_p.$ These points are exactly $6p$, the $N_p$-stabiliser of each one of them has order $p$, and they split into six orbits under the action of $N_p.$ It follows that the signature of the action of $N_p$ is $(\gamma; p^6).$ The Riemann-Hurwitz formula implies $\gamma=0$. As $N_p$ is a normal subgroup of $G_p$, one has that ${\bf r}$ induces an automorphism $\tilde{\bf r}$ of $Y_{p,t}$. Consider the associated regular covering map $$\pi': Y_{p,t} \cong \mathbb{P}^1 \to X_{p,t}/\langle N_p, {\bf r} \rangle \cong Y_{p,t}/\langle \tilde{\bf r}\rangle \cong \mathbb{P}^1.$$Observe that $\tilde{\bf r}$ can be assumed to be $\tilde{\bf r}(z)=\exp(2 \pi i /3)z$ and  the ramification values of $\pi$ form two $\langle \tilde{\bf r}\rangle$-orbits. It follows that the signature of the action of $\langle N, {\bf r} \rangle$ on $X_{p,t}$ is $(0; 3,3,p,p)$. Similarly, as ${\bf s}$ normalises $\langle N_p, {\bf r} \rangle$, it induces an automorphism $\tilde{{\bf s}}$ of $Z_{p,t}:=Y_{p,t}/\langle \tilde{{\bf r}}\rangle$ in such a way that $Z_{p,t}/\langle \tilde{\bf s} \rangle \cong X_{p,t}/G_p.$ As $\tilde{\bf s}$ permutes the branch values of $\pi'$,  the signature of the action of $G_p$ on $X_{p,t}$ is $(0; 3, 2p, 2p), (0;2,2,3,p)$ or $(0; p, 6,6).$  The latter signature is not realised as $G_p$ does not have elements of order six. The proof of statement (2) follows from the fact that the $G_p$-stabiliser of $$[1: \lambda: 1] \in X_{p,t} \mbox{ where }\lambda^p = -t \pm \sqrt{t^2 - t - 2}$$has order two. 
 
 Let $\Gamma$ be a Fuchsian group of signature $(0; 2,2,3,p)$ such that $X_{p,t}/G_p \cong \mathbb{H}/\Gamma.$ If the automorphism group of $X_{p,t}$ contains $G_p$ properly, then the corresponding covering map $$X_{p,t}/G_p \to X_{p,t}/\mbox{Aut}(X_{p,t})$$induces an inclusion  $\Gamma< \Gamma'$, where $\Gamma'$ is a Fuchsian group. As the dimension of $\Gamma$ is one, the dimension of $\Gamma'$ equals zero or one. However, as proved by Singerman in \cite{Sing72}, no one-dimensional Fuchsian group contains properly a Fuchsian group of signature $(0; 2, 2, 3, p)$ for $p\geqslant 5,$ showing that the dimension of $\Gamma'$ is zero. It follows that only finitely many exceptional members of $\mathcal{F}_p$ have automorphism groups different from $G_p$, proving statement (3).

\subsection{Proof of Proposition \ref{thetak}} \label{vgs} Assume that the action of $G$ on $X \in \mathscr{C}_p$ is represented by the surface-kernel epimorphism $\theta =(g_1, g_2, g_3, g_4).$ The involutions of $G$ are $$a^{2i} b^i s, a^i b^{2i} sr \mbox{ and } a^i b^{-i} sr^2 \mbox{ for } i \in \{0, \ldots, p-1\},$$and the elements of $G$ of order three are $a^l b^m r$ and $a^l b^m r^2$ for $l,m \in \{0, \ldots, p-1\}.$ As the involutions of $G$ are conjugate, after considering the braid transformation $\Phi_1$ one obtains that $\theta$ is topologically equivalent to either $$(s, a^{i} b^{-i} sr^2,  a^{l} b^{m} r, (a^{i-1+m} b^{-2i-1})^{-1}  ) \mbox{ or }( s, a^{i} b^{2i} sr,  a^{l} b^{m} r^2, (a^{-i-1-m} b^{i-1})^{-1}  ).$$The conjugation by $s$ allows us to see that these surface-kernel epimorphisms are equivalent. In addition, if we conjugate the latter by $b^{-i}$ then we obtain that 
    \begin{equation}\label{epimorfismo}
     \theta \sim   (s, sr, a^l b^m r^2, a^{l-m} b^l) \mbox{ where } l,m \in \{0, \ldots, p-1\}.
    \end{equation}There are two cases to consider.    

\s

{\bf 1.} Assume $m = l$. Then \eqref{epimorfismo} turns into $(s, sr, a^l b^l r^2, b^l),$  where $l \neq 0$. The automorphism $G$ given by $a\to a^{e}, \ b\to b^{e}, \ s\to s, \ r \to r$ where $el \equiv 1 \mbox{ mod }p$ shows that \eqref{epimorfismo} is topologically equivalent to $\theta' := (s, sr, abr^2, b).$

\s

{\bf 2.} Assume $m \neq l$. The automorphism of $G$ given by $a\to a^{e}, b\to b^{e}, s\to s,  r \to r$ where $e(l-m) \equiv 1 \mbox{ mod }p$ shows that \eqref{epimorfismo} is topologically equivalent to $\theta_k = (s, sr, a^k b^{k-1}r^2, ab^k)$ for some $k \in \{0, \ldots, p-1\}.$

\s

Finally, the transformation $\Phi_3 \circ \Phi_3$ coupled with the automorphism \begin{equation}\label{coca}a \mapsto a^{-1}, b \mapsto b^{-1}, r \mapsto r, s \mapsto s \end{equation}produce the equivalences $\theta_0 \sim \theta_1 \sim \theta'$ and $\theta_2 \sim \theta_{p-1},$ and the proof of the first statement is done. Note that $\Phi_3 \circ \Phi_3$ coupled with \eqref{coca} yield the equivalence $\theta_k \sim (s, sr, ab^{k}r^2, a^{1-k}b).$ Let $e$ such that $e(1-k) \equiv 1 \mbox{ mod }p.$ We can consider the automorphism of $G$ given by $ a\to a^{e}, \ b\to b^{e}, \ s\to s, \ r \to r$ to see that $\theta_k \sim \theta_{e}.$ We define the bijective map $$f: \{2, \ldots, p-2\} \to \{2, \ldots, p-2\} \text{\ such that \ } f(i)= (1-i)^{-1}$$where the inverses are taken modulo $p.$ Observe that  $f^{-1}(j)=1- j^{-1}.$ If it easy to see that $f$ has a fixed point if and only if $p\equiv 1 \mbox{ mod 3}$ (in such a case, $f$ has two fixed points) and the second statement follows.

\subsection{Proof of Proposition \ref{fanta}}We shall only give a step-by-step proof for the case $H=\langle a s\rangle \cong \mathbb{Z}_{2p},$ as the remaining cases are proved analogously. We refer to \cite[\S 3]{Rojas07} for the ramification of intermediate coverings.

We denote the regular covering maps given by the action of $G$ and $H$ on $X \in  \GP^k$ by $$\pi : X \to X/G \, \mbox{ and } \, \pi' : X \to X/H.$$Let $P_1, P_2, P_3$ and $P_4$ denote the (ordered) ramification values of  $\pi$. As the ramification values of $\pi$ that lie in $\pi^{-1}(P_3)$ have $G$-stabiliser of order $3$, the $H$-stabilizers of them are trivial. Thus, the ramification points of $\pi'$ are contained in the $\pi$-fiber of $P_1, P_2$ and $P_4.$

\s

The $\pi$-fiber of $P_1$ consists of $3p^2$ points and their $G$-stabilisers are conjugate to $\langle s \rangle.$ As the conjugacy class of $\langle s \rangle$ has length $3p$ coupled with the fact that $\langle as \rangle$ has a unique involution, we see that among such $3p^2$ points, exactly $p$ of them have non-trivial $H$-stabiliser. These points give rise to one ramification value of $\pi'$ marked with $2.$ The same argument shows that the ramification points in the $\pi$-fiber of $P_2$ give rise to one ramification value of $\pi'$ marked with $2.$ On the other hand, the $\pi$-fiber of $P_4$ consists of $6p$ points with $G$-stabiliser conjugate to $\langle ab^k\rangle.$ Note that the elements of order $p$ of $\langle as \rangle$ are $b^i$ for $i=1, \ldots, p-1.$ 

\s

{\bf (1)} If $k \neq 1$ then  $\langle ab^k \rangle$ and $\langle b \rangle$ are nonconjugate and therefore the ramification points of $\pi$ are not ramification points of $\pi'.$ Thus, the signature of the action of $\langle as \rangle$ is $(\gamma; 2,2)$ where, by the Riemann Hurwitz formula, $\gamma$ equals $p-1.$

\s

{\bf (2)} If $k = 1$ then, as $\langle ab\rangle$ and $\langle b \rangle$ are conjugate, among the $6p$ ramification points of $\pi$, exactly $2p$ of them are ramification points of $\pi'.$ Such points yield $p$ ramification values of $\pi'$ marked with $p.$ Thus, the signature of the action of $\langle as \rangle$ is $(\gamma; 2,2, p, \stackrel{p}{\ldots}, p)$ and the value of $\gamma$ follows from the Riemann Hurwitz formula.

\subsection{Proof of Theorem \ref{cig}}

We recall that the automorphism group $G'$ of $\mathbf{F}_{2p}$ is isomorphic to  \begin{equation}\label{ppf}
 \langle A, B, S, R :  A^{2p},B^{2p},R^3,S^2,(SR)^2,[A,B],1, RAR^{-1}(AB), RBR^{-1}A^{-1}, SAS(AB), [S,B]\rangle,
\end{equation}and that the signature of this action is $(0;2, 3, 4p)$. Let $\Delta'$ be a Fuchsian group of this signature  $$\Delta' = \langle z_1, z_2, z_3:z_1^2 = z_2^{3} = z_3^{4p}= z_1z_2z_3 =1\rangle.$$ The action of  $G'$ on $\mathbf{F}_{2p}$ is represented by the surface-kernel epimorphism $$\Theta: \Delta' \to G' \mbox{ given by }\Theta(z_1)=SR^2, \Theta(z_2)=AR^2, \Theta(z_3)=ABSR.$$Observe that the elements  $z_2^2 z_1 z_2, z_2z_1z_2^2, z_1z_2^2z_1$ and  $z_3^{-4}$ generate a subgroup $\Gamma$ of $\Delta'$ isomorphic to a Fuchsian group of signature $(0; 2, 2, 3, p)$. In addition, $$\Theta(z_2^2 z_1 z_2)=A^2BS, \ \Theta(z_2 z_1 z_2^2)=SR, \ \Theta(z_1 z_2^2 z_1)=ABR^2, \ \Theta(z_3^{-4})=A^{-2}$$
and $\Theta(\Gamma) = H:=\langle A^2 , B^2, S, R \rangle \cong \G.$ Consequently, the restriction of $\Theta$ to $\Gamma$  is a surface-kernel epimorphism  $\theta: \Gamma \to H$ that represents the action of $H$ on $\mathbf{F}_{2p}$. Note that, after conjugating $\Theta$ by $A^{-1}$, we have that $\Theta$ is equivalent to $\theta_0=(S, SR,  B^2 R ^2, A^{-2}).$ Finally, this last surface-kernel epimorphism is, as shown in the proof of Proposition \ref{thetak}, topologically equivalent to $\theta_1$ and the proof of the first statement is done.

\s

{\it Claim.} $\mathbf{F}_{2p}$ belongs to $\GP^1$ only.

\s

Assume that $\mathbf{F}_{2p}$ belongs to an equisymmetric stratum of $\GP$ determined by a surface-kernel epimorphism topologically inequivalent to $\theta_1.$ This implies, in particular,  that $G'=\mbox{Aut}(\mathbf{F}_{2p})$ has a subgroup $H'$ such that $$H' \cong H=\langle A^2 , B^2, S, R\rangle \mbox{ but $H$ and $H'$ are nonconjugate}.$$ As $\langle A^2, B^2 \rangle$ is the unique (normal) Sylow $p$-subgroup of $G',$ we have that necessarily $H'=\langle A^2, B^2 \rangle \rtimes \langle S', R' \rangle$ where $S', R'$ satisfy $\langle S', R'\rangle \cong \mathbf{D}_3.$ Now, as each element of $G'$ of order three is conjugate to either $R$ or $R^2$ we obtain that $H'$ is conjugate to $\langle A^2, B^2 \rangle \rtimes \langle R, S'\rangle.$ After some routine computations, we see that the involution $S'$ is necessarily equal to $S, SR$ or $SR^2$. In each case we obtain that $H$ is conjugate to $H'$, proving the claim.

\s

Finally,  the second statement follows from the connectedness of $\mathcal{F}_p^*$.

\subsection{Proof of Proposition \ref{con3}}Assume that there exists $X \in \GP$ with automorphism group of order $18p^2.$ The maximality of the signature $(0;2,2,3,p)$ implies that $\mbox{Aut}(X)$ acts on $X$ with triangle signature; say, $(0; m_1, m_2, m_3).$ Then the Riemann-Hurwitz formula implies that $$\tfrac{1}{m_1} +\tfrac{1}{m_2}+\tfrac{1}{m_3} = \tfrac{7p+3}{9p}.$$A routine computation shows that this equation admits solutions if and only if $p = 5$. However, according to \cite{conder}, there does not exist a Riemann surface of genus $36$ endowed with a group of automorphisms of order 450, contradicting our assumption.

\subsection{Proof of Proposition \ref{prol2}}

The proof of the proposition is a consequence of the following lemmata.

\begin{lemm} \label{l1} Let $p \geqslant 5$ be a prime number.  If a group $H$ of order $12p^2$ acts on some member of $\mathscr{C}_p$ then the signature of the action is $(0; 2,6,2p)$.
\end{lemm}
\begin{proof} Let $X \in \GP$ and let $\pi : X \to X/G$ be the regular covering map given by the action of $G$ on $X.$ The fact that $K=G/H \cong \mathbb{Z}_2$ acts on $X/G$ implies that the branch values of $\pi$ marked with 3 and $p$ must be fixed under the action of $K$, and that the two branch values of $\pi$ marked with $2$ must form a $K$-orbit. It follows that, as $X/H \cong (X/G)/K,$  the signature of the action of $H$ on $X$ must be $(0; 2,6,2p)$.
\end{proof}
\begin{lemm}Let $p \geqslant 5$ be a prime number different from $11$. If a group $H$ of order $12p^2$ acts on some member of $\mathscr{C}_p$ then $H \cong H_1$ or $H \cong H_2.$
\end{lemm}

\begin{proof}
As the result holds for $p=5$ (see, for instance, \cite{conder}), we assume $p \geqslant 7$ and different from 11. Observe that there is a unique Sylow $p$-subgroup $P$ of $G$ which is isomorphic to $\Z_p ^2.$ By the 
Schur-Zassenhaus, we have that $$H \cong P \rtimes L \, \mbox{ where } \, L:=G/P \mbox{ has order }12.$$

Consider the surface-kernel epimorphism $\Theta : \Delta \to H$ representing the action of $H$ on $X$, where, by Lemma \ref{l1}, $\Delta$ is the Fuchsian group 
\begin{equation}\label{epit} \Delta= \langle z_1, z_2,  z_3 : z_1^2 = z_2^6  =z_3^{2p} = z_1z_2z_3=1 \rangle.\end{equation}If $\tau: G \to L$ is the canonical projection then the composite map $f:=\tau \circ \Theta : \Delta \to L$ is an epimorphism of groups. As $p \neq 2,3$, the fact that $f(z_1)^2=f(z_3)^{2p}=id$ shows that $f(z_1)$ and $f(z_2)$ equal the identity or are involutions of $L.$ Note that this implies that $L$ is nonabelian. The fact that $\mathbf{A}_4$ cannot be generated by two involutions, coupled with the fact that $\Z_4 \rtimes_2 \Z_3$ has a unique involution allow us to conclude that $L \cong \mathbf{D}_6.$ Then
 $$H \cong \Z_p^2 \rtimes \mathbf{D}_6 \cong \Z_p^2 \rtimes (\mathbf{D}_3 \times \Z_2) \cong (\G) \rtimes \Z_2$$and therefore  $H= \langle A, B, S, R, C\rangle$ with partial presentation 
$$A^p ,B^p,S^2,R^3,(SR)^2,C^2,  [A,B],
    SASAB,  [B,S],RAR^2AB, RBR^2A^{-1},[C, R],[C,S].
$$We only need to determine $CAC$ and $CBC.$ As $P$ is normal in $H$ we have that $K:=\langle A, B, C\rangle \cong \langle A,B\rangle \rtimes \langle C \rangle \cong \Z_p^2 \rtimes \Z_2.$ Up to isomorphism, the action of $C$ on $\langle A, B \rangle$ is given by $$(A, B) \mapsto (A,B), \, (A, B) \mapsto (A^{-1},B^{-1}) \, \mbox{ or }\, (A, B) \mapsto (A,B^{-1}).$$
We observe that the latter possibility is not realised. Indeed, in such a case, $R(BC)R^2=AC,$ but $BC$ has order $2$ whereas $AC$ has order $2p$. The proof of the lemma follows.
\end{proof}

%

\begin{lemm}\label{Z11}
There is a unique, up to automorphisms of $H_1$, surface-kernel epimorphism $\Delta \to H_1.$
\end{lemm}

\begin{proof} Consider the automorphism $\Phi$ of $H$ given by $A\mapsto A, B\mapsto B,  R\mapsto R, S\mapsto  SC, C\mapsto C$ and observe that if $g$ is an involution then $$g \notin  \langle A, B, S, R\rangle \, \implies \, \Phi(g)  \in  \langle A, B, S, R\rangle.$$Let $\Theta=(g_1, g_2, g_3)$ be a surface-kernel epimorphism of $H_1$ of type $(0;2, 6, 2p)$. By the previous observation, we may assume that $g_1 \in \langle A, B, S, R \rangle.$  Further, after a suitable conjugation, we  can assume $g_1=S.$ In addition, as the elements of order 6 of $H_1$ are
$A^i B^j RC$ and $A^i B^j R^2C$ where $0 \leq i, j \leq p-1,$ it follows that $\Theta$ is $H_1$-equivalent to $$(S,A^i B^j RC, (SA^i B^j RC)^{-1}) \, \mbox{ or }\, (S, A^i B^j R^2C, (SA^i B^j R^2C)^{-1}) \, \mbox{ for some }0 \leq i, j \leq p-1.$$ The conjugation by $S$ shows that the two possibilities above are $H_1$-equivalent. In addition, after conjugating by $B^j$, one sees that $\Theta$ is $H_1$-equivalent to $(S, A^iRC, (SA^iRC)^{-1})$ where $i \neq 0.$ Finally, if $e \in \mathbb{Z}_p$ satisfies $ie \equiv 1 \mbox{ mod }p$ then the automorphism of $H_1$ given by $A\mapsto A^{e}, B\mapsto B^{e}, S\mapsto S, R\mapsto R$ shows that $\Theta$ is $H_1$-equivalent to $\Theta_1= (S, ARC, B^{-1}SRC).$ 
\end{proof}

\begin{lemm}
There is a unique, up to automorphisms of $H_2$, surface-kernel epimorphism $\Delta \to H_2.$
\end{lemm}
\begin{proof}
 A routine computation shows that the involutions of $H_2$ that do not belong to $\langle A, B, S, R\rangle$ are of the following types:
$$\mbox{\bf Type I:} \, B^iSC, A^iSRC, A^iB^iSR^2C \,\,\, \mbox{ and } \,\,\,\mbox{\bf Type II:} \,A^iC, B^iC, A^iB^jC.$$ Let $\Theta=(g_1, g_2, g_3)$ be a surface-kernel epimorphism of $H_2$ of type $(0;2, 6, 2p)$ and assume that $g_1 \notin \langle A, B, S, R\rangle.$ Notice that $g_1$ cannot be of type II. Indeed, as the elements of $H_2$ of order $6$ are $A^i B^j RC, A^i B^j R^2C$ where $0 \leqslant i, j \leq p-1,$ one has that an element of order 6 and an involution of type II do not generate $H_2,$ contradicting the surjectivity of $\Theta.$ Besides, if $g_1$ is of type I then $\Phi(g_1) \in \langle A, B, S, R\rangle,$  where $\Phi$ is the automorphism of $H_2$ given by $$A\mapsto A^{-1-2z}B^{-2w}, B\mapsto A^z B^w,  R\mapsto A^2B^2R, C \mapsto A^{-2w+2}B^{-2w+2z+2}C, S\mapsto SC
$$with $z=-2(3^{-1})$ y $w=-3^{-1}$. As a result, up to automorphism of $H_2$, we can assume $g_1\in \langle A, B, S, R\rangle.$ Further, we can assume $g_1=S$ and therefore $\Theta$ is $H_2$-equivalent to $$(S, A^i B^j RC, (SA^i B^j RC)^{-1}) \mbox{ or } (S, A^i B^j R^2C, (SA^i B^j R^2C)^{-1})$$ Now,  as in Lemma \ref{Z11}, we can conclude that $\Theta$ is $H_2$-equivalent to $\Theta_2=(S, ARC, BSRC).$ 
\end{proof}

We denote by $Z_i$ the Riemann surface $\mathbb{H}/\mbox{ker}(\Theta_i)$ given in the two lemmata above. 

\begin{lemm}
$Z_i \in \GP^i$ and $\mbox{Aut}(Z_i) \cong H_i$ for $i=1,2.$
\end{lemm}
 \begin{proof} Let $\pi': Z_1 \to Z_1/H_1$ and $\pi:Z_1 \to Z_1/G$ be the regular covering maps given by the action of $H_1$ and $G$ on $Z_1$ respectively. We denote by $q_1, q_2, q_3$ the (ordered) ramification values of $\pi'$. 

\s

{\bf (1)}  The $H_1$-stabiliser of each point in the $\pi'$-fiber of $q_1$ is conjugate to $\langle S \rangle$. The fact that $G$ is normal in $H_2$ shows that the $6p^2$ points in the $\pi'$-fiber of $q_1$ have $G$-stabiliser of order two, and hence they give rise to two branch values of $\pi,$ marked with $2$ each.

\s

{\bf (2)}  The $H_1$-stabiliser of each point in the $\pi'$-fiber of $q_2$ is conjugate to $\langle ARC \rangle.$ Note that $g \langle ARC \rangle g^{-1} \cap G \cong \mathbb{Z}_3$ for each  $g \in H_2.$ Thus, the $2p^2$ points in the $\pi'$-fiber of $q_2$ give rise to one branch point of $\pi,$ marked with $3$.

\s

{\bf (3)} The $H_1$-stabiliser of each point in the $\pi'$-fiber of $q_3$ is conjugate to $\langle  B^{-1}SRC \rangle.$ Note that $g \langle  B^{-1}SRC \rangle g^{-1} \cap G \cong \mathbb{Z}_p$ for each $g \in H_2.$ Thus, the $6p$ points in the $\pi'$-fiber of $q_3$ give rise to one branch point of $\pi,$ marked with $p$.

\s
    
Thus, the signature of the action of $G$ on $Z_1$ is $(0; 2,2,3,p)$, showing that $Z_1 \in \GP.$ 

\s

Now, consider the surface-kernel epimorphism $\Theta_1$ that represents the action of $H_1$ on $Z_1$
$$\Theta_1:\Delta \to H_1 \,\, \mbox{ given by }\,\, \Theta_1=(S, ARC, B^{-1}SRC),$$where $\Delta$ is presented in \eqref{epit}. Observe that the elements $\alpha_1:=z_1, \alpha_2:=z_2z_1z_2^{-1}, \alpha_3:=  z_2^2$ and $\alpha_4:=  z_3^2$ generate a subgroup $\Gamma$ of $\Delta$ isomorphic to a Fuchsian group of signature $(0; 2,2,3,p).$  The restriction of $\Theta$ to $\Gamma$ is given by $$(\alpha_1, \alpha_2, \alpha_3, \alpha_4) \mapsto (S, SR, B^{-1}R^{2}, A),$$which is topologically equivalent to the surface-kernel epimorphism $\theta_1$ of Proposition \ref{thetak}. Hence, $Z_1 \in \GP^1$. Finally, the fact that the signature $(0; 2,6,2p)$ is maximal for each $p \geqslant 5$  (see \cite{Sing72}) implies that $\mbox{Aut}(Z_1) \cong H_1.$

For the sake of brevity, we omit the proof for $Z_2$ as it is analogous to that of $Z_1.$ 
\end{proof}
\begin{lemm}
$Z_1$ and $Z_2$ do not belong to $\FP^*.$

\end{lemm}

\begin{proof}

If $Z_1$ belongs to $\FP^*$ then there exists a linear automorphism of $\mathbb{P}^2$ given by ${\bf c} \in \mbox{GL}(3, \mathbb{C})$ of order $2$ such that $\langle {\bf a,b,s,r,c} \rangle \cong H_1$, where ${\bf a,b,s,r}$ are given in the statement of Proposition \ref{noprimo1}. A computation shows that $$[{\bf a, c}]=[{\bf b, c}]=[{\bf r, c}]=[{\bf s, c}]=\mbox{id}_3$$implies that ${\bf c}=\mbox{id}_3.$ The same argument permits us to conclude for $Z_2.$
\end{proof}

\subsection{Proof of Theorem \ref{fullg}}

As pointed out in \S\ref{state}, if the automorphism group of $X_{p,t} \in \mathcal{F}_p^*$ is not isomorphic to $\G$ then one of the following statements hold: $${\mbox{ $X_{p,t} \cong \mathbf{F}_{2p}$, or  $\mbox{Aut}(X_{p,t})$ has order $12p^2$, or 
 $\mbox{Aut}(X_{p,t})$ has order $18p^2.$}}$$ By the lemmata proved in the previous section, if the second case holds then $X_{p,t}$ is isomorphic to $Z_1$ or $Z_2$. However, such Riemann surfaces do not belong to $\FP^*.$ In addition, by Proposition \ref{con3}, the latter possibility is not realised.  The proof of the theorem follows.

\subsection{Proof of Theorem \ref{clasi}}

The proof of the theorem is a consequence of the lemmata proved in this section.

\begin{lemm}\label{supW}
Let $W$ be a compact Riemann surface of genus  $(p-1)(2p-1)$ where $p \geqslant 5$ is a prime number. If $W$ has a group of automorphism $G$ isomorphic to $\G$ then the signature of the action of $G$ on $W$ is $(0;2,2,3,p)$ or $(0; 3, 2p, 2p)$.
\end{lemm}
\begin{proof}
Let $\pi: W \to W/H$ be the regular covering map given by the action of $N = \langle a, b \rangle \cong \Z_p^2$ on $W$. The signature of the action of $N$ is of the form $(\gamma; p, \stackrel{m}{\ldots}, p)$ for some  $m \geqslant 0$ such that  $$4p^2 - 6p =p^2(2\gamma-2 + B) \mbox{ where }  B=m(1-\tfrac{1}{p}).$$
The fact that $p \geqslant 5$ implies that $\gamma=0$ and therefore $m=6.$ Now,  $G/N \cong \mathbf{D}_3$ must act on $W/N\cong \mathbb{P}^1$ keeping invariant the set of six branch values of $\pi.$ Thus, the signature of the action of $G$ on $W$ is
$(0;6,6,p), (0;2,2,3,p)$ or $(0;3,2p,2p),$ and the proof follows after noticing tht $G$ does not have elements of order $6.$
\end{proof}

Let $\Delta$ be a Fuchsian group of signature 
$(0; 3, 2p, 2p)$ presented as $$\langle y_1,y_2, y_3 :y_1^3=y_2^{2p}= y_3^{2p}=y_1y_2y_3=1 \rangle.$$

\begin{lemm} Let $p\geqslant 5$ be a prime number. If $Y$ is a compact Riemann surface endowed with a group of automorphisms $G$ isomorphic to $\G$ acting with signature $(0; 3, 2p, 2p)$, then the action is described by the surface-kernel epimorphism  $$\varphi_n : \Delta \to G, \,\, \varphi_n= (a^{1-n}br^2, bsr, b^{n}sr^{2}) \mbox{ for some } 1 \leqslant n \leqslant p-1.$$
\end{lemm}
\begin{proof} Let $\varphi=(g_1, g_2, g_3)$ be a surface-kernel epimorphism of $G$ of signature $(0;3, 2p, 2p).$ The elements of order $2p$ of $G$ are of three types: $a^{j} b^i s$ with $j \neq 2i$, $a^i b^{j} sr$ with $j \neq 2i$  and $a^i b^{j} sr^2$ with $j \neq -i$, where $0 \leq i, j \leq p$. Observe that the product of two elements of order $2p$ of the same type has order $p.$ Thus, after a suitable conjugation, we see that $\varphi$ is $G$-equivalent to  
\begin{equation}\label{nm}((a^i b^m sr a^lb^j sr^2)^{-1}, a^i b^ms r,a^lb^jsr^2)\end{equation}where $i,m,l,j \in \{0, \ldots, p-1\}$ satisfy $m \neq 2i$ and $j \neq -l.$ After conjugating \eqref{nm} by $a^{-i-l}b^{-i}$, we obtain that $\varphi$ is $G$-equivalent to the surface-kernel epimorphism
$$((b^{m-2i}srb^{l+j}sr^2)^{-1},b^{m-2i}sr,b^{l+j}sr^2).$$ We now apply the automorphism $a\mapsto a^{e}, b\mapsto b^{e}, s\mapsto s, r \mapsto r$ where $(m-2i)e \equiv 1 \mbox{ mod }p$ to see that $\varphi$ is $G$-equivalent to
  $\varphi_n,$ where $n=(l+j)m$, as desired.
\end{proof}

\begin{lemm} Set $Y_n := \HH/\mbox{ker}(\varphi_n)$. The following statements hold.\mbox{}
\begin{enumerate}
\item $Y_{1} \cong Z_1$ and $Y_{p-1} \cong Z_2.$
\item There are exactly $\tfrac{p-3}{2}$ pairwise nonisomorphic Riemann surfaces among $\{Y_2, \ldots, Y_{p-2}\}$ and the automorphism group of each one of them is isomorphic to $G.$
\end{enumerate}
\end{lemm}

\begin{proof} Assume that the automorphism group of $Y_n$ contains $G$ properly. Thus, following the results of \cite{Sing72}, the automorphism group of $Y_n$ acts with signature $(0; 2,6,2p).$ In addition, as proved in \cite[Case N8]{BCC}, there is an automorphism $\Phi$ of $G$ of order two such that $$\Phi(bsr)=b^nsr^2 \, \mbox{ and } \Phi(a^{1-n}br^2)=a^{1-n}br.$$The equalities above show, after a routine computation, that $\Phi(s)=a^{\tau}b^{\tau}s$ where $\tau=\tfrac{1}{n}-n+1-n^2$ (the inverse is taken modulo $p$). The fact that $\Phi(s)$ has order two implies that $\tau=0$ and therefore $n=\pm 1$. We then conclude that $\mbox{Aut}(Y_n)\cong G$ for each $n \in \{2, \ldots, p-2\}.$ If $N(\Delta)$ is the normaliser of $\Delta$ then $N(\Delta)/\Delta \cong \mathbb{Z}_2$ acts on $\{\varphi_2, \ldots, \varphi_{p-2}\}$ by conjugation with orbits $\{\varphi_n, \varphi_{1/n}\}$. Thus, the desired number of isomorphism among $Y_2, \ldots, Y_{p-2}$ follows from the fact that this action identifies Riemann surfaces that are isomorphic.
\s

Now, we prove that $Y_{1} \cong Z_1.$ We recall that, by Lemma \ref{Z11}, $Z_1$ is the unique compact Riemann surface with an action of $H_1=\langle A, B, S, R, C \rangle \cong (\G) \times \mathbb{Z}_2$ with signature $(0; 2,6,2p).$ If  $$\Delta'= \langle z_1, z_2,  z_3 : z_1^2 = z_2^6  = z_3^{2p} = z_1z_2z_3=1 \rangle$$ then the action is given by the surface-kernel epimorphism $\Theta_1= (S, ARC, B^{-1}SRC).$ Note that $\Delta_0=\langle z_2^{2}, z_3, z_1z_3z_1^{-1}\rangle$ is a Fuchsian group of signature $(0; 3,2p,2p)$ and the restriction of $\Theta_1$ to it is given by $$\Delta_0 \to \Theta_1(\Delta_0) \mbox{ where } (z_2^{2}, z_3, z_1z_3z_1^{-1}) \mapsto (B^{-1}R^2, B^{-1}SRC, B^{-1}SR^2C)$$If we write  $\tilde{A}:= A,  \tilde{B}:=B,\tilde{S}:=SC,\tilde{R}:=R,$ then $\Theta_1(\Delta_0)=\langle \tilde{A}, \tilde{B},\tilde{S},\tilde{R}\rangle \cong \G.$ Thus, the restriction above yields an action of $\G$ with signature $(0; 3,2p,2p)$ represented by $$ (\tilde{B}^{-1}\tilde{R}^2,\tilde{B}^{-1}\tilde{S}\tilde{R},\tilde{B}^{-1}\tilde{S}\tilde{R}^2),$$which is $G$-equivalent to $\varphi_{1}=(br^2,bsr,bsr^2).$ This shows that $Z_1 \cong Y_1.$ The proof that $Z_2 \cong Y_{p-1}$ is analogous.
\end{proof}

\section{Final remarks and comments}\label{rems}

\subsection{The relationship with generalised Fermat curves}\label{CFG}

Let $n,k \geqslant 2$ be integers such that $(n-1)(k-1)  \geqslant 3$. We recall that a {\it generalised Fermat curve $S$ of type $(k,n)$} is a compact Riemann surface admitting a group of automorphisms $H$ isomorphic to $\mathbb{Z}_k^n$ in such a way that $S \to S/H \cong \mathbb{P}^1$ ramifies over $n+1$ values marked with $k.$ The proof of Lemma \ref{supW} provides an interesting relationship between the Riemann surfaces we are considering and the generalised Fermat curves. 
More precisely, let $W$ be a compact Riemann surface as in Theorem \ref{clasi}. As $W$ admits a branched $\mathbb{Z}_p^2$-covering map onto the Riemann sphere ramified over six values marked with $p$, according to the results of  \cite{GHL}, one has that $W \cong S/K$ where $S$ is a generalised Fermat curve of type $(p,5)$ and $K \cong \mathbb{Z}_p^3$ is a group of automorphisms of $S$ acting freely on $S$. See also \cite{HKLP}. This relation could allow us to employ well-known facts concerning generalised Fermat curves to obtain information about our Riemann surfaces and generalisations of them (for instance, different algebraic descriptions, isogeny decompositions of their Jacobian varieties, among other aspects). We plan to consider this relationship in a forthcoming work. 

\subsection{Singular members of $\FP$}\label{singular}

The pencil of Kuribayashi-Komiya quartics $\mathcal{F}_2$ has four singular members. Following \cite[Proposition 4.7]{KFT} such singular members are
\begin{enumerate}
\item The reducible curve corresponding to a double
conic: $X_{2,2}:(x^2+y^2+z^2)^2=0.$

\item The reducible curve corresponding to the union of four lines$$X_{2,-2}: (x-y+z)(x+y-z)(-x+y+z)(x+y+z)=0.$$

\item The reducible curve corresponding to the union
of two conics, $$X_{2,-1}:(x^2 +\omega y^2 +\omega^2z^2)(x^2 +\omega^2y^2 +\omega z^2)=0, \mbox{ where } \omega=\mbox{exp}(2 \pi i /3).$$

\item The irreducible curve $X_{2,\infty}: x^2y^2 + y^2z^2 + z^2x^2 = 0$ with double points at $ [1:0:0], [0:1:0]$ and $[0:0:1]$.

\end{enumerate}The singular curves $X_{2,-2}, X_{2,-1}$ and $X_{2, \infty}$ are stable and then represent points in the boundary of the Deligne-Mumford compactification of $\mathscr{M}_3.$ Furthemore, as noticed in \cite[Remark 4.8]{KFT}, the non-stable curve $X_{2,2}$ corresponds to the hyperelliptic Riemann surface represented by  $y^2=x^8-14x^4+1.$ This is the unique Riemann surface of genus three with automorphism group isomorphic to $\mathbf{S}_4 \times \mathbb{Z}_2$.  
It is not difficult to see that $\mathscr{C}_2-\mathcal{F}^*_2=\{H\}.$

\s

In the case of the pencil of generalised Kuribayashi-Komiya curves $\FP$ we also have exactly four singular members, but among them only one is an stable curve. Concretely, $$X_{p,-1}: x^{2p}+y^{2p}+z^{2p}-x^p y^p -x^p z^p -y^p z^p=0$$has exactly $p^2$ singular points, given by $[\zeta_p^i:\zeta_p^j:1]$ where $\zeta_p=\mbox{exp}(2 \pi i/p)$ and $0 \leqslant i,j <p$. All such singular values are nodes. By the semistable reduction theorem due to Deligne and Mumford \cite{DM}, there is a stable curve (which may be smooth) representing the nonstable singular members of $\FP.$ The fact that $Z_1 \in \GP^1-\FP^*$ shows that one of them must be isomorphic to $Z_1$.

\subsection*{Acknowledgements} The authors are grateful to Rub\'en A. Hidalgo (Universidad de La Frontera) for his helpful comments and suggestions.

\subsection*{\bf Conflict of Interest and Data Availability} The authors have no conflict of interest to declare that is relevant to this article. All data generated or analysed during this study are included in this manuscript.

\end{document}